\documentclass{amsart}
\usepackage{amsmath,amssymb}
\usepackage{graphicx,subfigure}
\newtheorem{theorem}{Theorem}[section]
\newtheorem{proposition}[theorem]{Proposition}

\newtheorem{lemma}[theorem]{Lemma}

\theoremstyle{definition}

\theoremstyle{remark}

\numberwithin{equation}{section}

\newcommand{\al}{\alpha}
\newcommand{\be}{\beta}
\newcommand{\de}{\delta}
\newcommand{\ep}{\epsilon}

\newcommand{\la}{\lambda}
\newcommand{\om}{\omega}
\newcommand{\si}{\sigma}

\newcommand{\De}{\Delta}

\newcommand{\La}{\Lambda}

\def\hg{\widehat g}

\newcommand{\tu}{\widetilde{u}}
\newcommand{\tbe}{\widetilde{\be}}

\newcommand{\tY}{\widetilde{Y}}

\def\RR{\mathbb{R}}
\def\BB{\mathbb{B}}
\def\ZZ{\mathbb{Z}}

\def\TT{\mathbb{T}}
\def\MM{{\mathbb{M}^3}}

\renewcommand\SS{\mathbb{S}}
\newcommand\Bu{{\bar u}}

\newcommand{\cH}{{\mathcal H}}

\newcommand{\cS}{{\mathcal S}}

\newcommand{\pd}{\partial}
\newcommand\minus\backslash

\newcommand\lan\langle
\newcommand\ran\rangle

\newcommand{\e}{{e}}

\DeclareMathOperator\Div{div}

\DeclareMathOperator\dist{dist}

\renewcommand\leq\leqslant
\renewcommand\geq\geqslant
\newlength{\intwidth}

 \DeclareMathOperator\curl{curl}

\begin{document}

\title[Knotted
  structures in Beltrami fields on the torus and the sphere]{Knotted
  structures in high-energy Beltrami fields on the torus and the sphere}

\author{Alberto Enciso}
\address{Instituto de Ciencias Matem\'aticas, Consejo Superior de
  Investigaciones Cient\'\i ficas, 28049 Madrid, Spain}
\email{aenciso@icmat.es, dperalta@icmat.es, fj.torres@icmat.es}

\author{Daniel Peralta-Salas}

\author{Francisco Torres de Lizaur}

%
%
\begin{abstract}
Let $\cS$ be a finite union of (pairwise disjoint but possibly knotted and linked)
closed curves and tubes in the round sphere $\SS^3$ or in the flat
torus $\TT^3$. In the case of the torus, $\cS$ is further assumed to
be contained in a contractible subset of $\TT^3$. In this paper we show that for any
sufficiently large odd integer $\la$ there exists a Beltrami field on
$\SS^3$ or $\TT^3$ satisfying $\curl u=\la u$ and with a collection of
vortex lines and vortex tubes given by~$\cS$, up to an ambient diffeomorphism.
\end{abstract}
\maketitle

\section{Introduction}

An incompressible fluid flow in $\RR^3$ is described
by its velocity field $u(x,t)$, which is a time-dependent vector field satisfying the Euler equations
\[
\pd_t u+ (u\cdot\nabla)u=-\nabla P\,,\qquad \Div u=0
\]
for some pressure function $P(x,t)$. When the velocity field does not
depend on time, the fluid is said to be {\em stationary}\/. This paper
concerns stationary solutions of the Euler equations, which describe equilibrium
configurations of the fluid.

A central topic in topological fluid mechanics, which can be traced
back to Lord Kelvin in the XIX~century~\cite{Kelvin}, concerns the existence
of knotted stream and vortex structures in stationary fluid
flows. The most relevant of these structures are the stream lines, vortex lines and
vortex tubes of the fluid. We recall that a {\em stream line} and a
{\em vortex line}\/ are simply a trajectory (or integral curve) of the
velocity field $u$ and the vorticity $\om:=\curl u$, respectively,
while a {\em vortex tube}\/ is the interior domain bounded by an
invariant torus of the vorticity. The existence of topologically
complicated stream and vortex lines is a central topic in the
Lagrangian theory of turbulence and in magnetohydrodynamics, and has
been studied extensively in the last decades (see e.g.~\cite{Kh05,Mo14} for recent accounts of the subject).

Our understanding of the set of stationary states of the Euler
equations in three dimensions is much more limited than in the two-dimensional situation~\cite{CS12,Nad13}. In particular, the existence of stationary solutions in $\RR^3$ having stream lines, vortex lines and
vortex tubes that are knotted and linked in 
arbitrarily complicated ways has been established only very recently~\cite{Annals,Acta,EMS}. Following a suggestion of
Arnold~\cite{Ar65,AK} related to his celebrated structure theorem, to
prove these results one does not consider just any kind of solutions to
the stationary Euler equations but a very particular class that are
called Beltrami fields. A {\em Beltrami field}\/ in~$\RR^3$ is a vector field
satisfying the equation
\begin{equation}\label{Beltrami}
\curl u=\la u
\end{equation}
for some nonzero constant~$\la$. Notice that stream lines and vortex lines coincide in the
case of a Beltrami field, and that a Beltrami field is automatically
smooth (even real analytic) by the elliptic regularity theory.

The stationary solutions in~$\RR^3$ that one can construct using the techniques
in~\cite{Annals,Acta} fall off at infinity as $1/|x|$, this decay
being sharp for Beltrami fields but not fast enough for the
velocity to be in the energy space $L^2(\RR^3)$. In fact, the
incompressibility condition ensures that there are no Beltrami fields
in $\RR^3$ with finite energy even if the proportionality factor $\la$
is allowed to be nonconstant, as has been recently shown in~\cite{Nad14,CC}. 

On the contrary, Beltrami fields in a closed Riemannian 3-manifold $M$ (or a bounded domain of $\RR^3$) are
stationary solutions to the Euler equations that do have finite
energy. If $\cS$ is a union of
(possibly knotted and linked) closed curves and embedded tori in the
3-sphere, in this setting one can use contact topology to show~\cite{EG00} that
there is a Riemannian metric~$g$ on the sphere
with an associated Beltrami field $u$ having a collection of vortex
lines and vortex tubes given precisely by~$\cS$. The main ideas of the proof
are that the Reeb field of a contact form is in fact a Beltrami
field in some adapted metric and that one can indeed construct contact
forms on the sphere whose Reeb fields have the collection
of periodic trajectories and invariant tori given by~$\cS$. Notice
that, as it is a Reeb vector field, a Beltrami field obtained in this fashion
does not vanish. Conversely, any nonvanishing Beltrami field on the
sphere is the Reeb vector field of some contact form, so in particular it must possess a closed vortex line~\cite{Ho93}.

Our goal in this paper is to establish the existence of knotted and
linked vortex structures in Beltrami fields on compact manifolds with
a {\em fixed}\/ Riemannian metric. Specifically, we will consider
Beltrami fields in the flat 3-torus $\TT^3$ and in the unit 3-sphere
$\SS^3$; in fact, the former is the most fundamental space considered
in the fluid mechanics literature other than~$\RR^3$ and the latter is
perhaps the simplest example of a closed Riemannian
3-manifold from a geometric point of view.

It is worth emphasizing that, for a fixed Riemannian structure, the
problem is much more rigid than when one can freely choose a metric
adapted to the geometry of the set of lines and tubes that one aims to
recover from the trajectories of a Beltrami field. An obvious reason
is that, analytically, Beltrami fields in a closed Riemannian manifold
arise as eigenfields of the curl operator, which defines a
self-adjoint operator with discrete spectrum and a dense domain in the space of
divergence-free $L^2$~fields. In the context of spectral theory, the
proportionality constant $\la$, or rather its absolute value, can be
thought of as the {\em energy}\/ of the Beltrami field, although of
course it is in no way related to the $L^2$~norm of the latter.

Our main theorem asserts that there are ``many'' Beltrami fields~$u$ in the sphere
and in the torus with vortex lines and vortex tubes of any link
type. Furthermore, these structures are {\em structurally stable}\/ in
the sense that any vector field on the torus or the sphere which is
sufficiently close to~$u$ in the $C^4$~norm and which
preserves some smooth volume measure will also have this collection of
periodic trajectories and invariant tori, up to a diffeomorphism. 
To state this result precisely, let us call a {\em tube}\/
the closure of a domain (in $\SS^3$ or $\TT^3$) whose boundary is an embedded
torus. Throughout, diffeomorphisms are of class
$C^\infty$, curves are all assumed to be non-self-intersecting, and we
will agree to say that an integer is large when it is large in absolute value.

\begin{theorem}\label{T.main}
Let $\cS$ be a finite union of (pairwise disjoint, but possibly knotted and linked) closed curves and tubes in
$\SS^3$ or $\TT^3$. In the case of the torus, we also assume that $\cS$
is contained in a contractible subset of~$\TT^3$. Then for any large
enough odd integer $\la$ there exists
a Beltrami field~$u$ satisfying the equation $\curl u=\la u$ and a diffeomorphism $\Phi$ of
$\SS^3$ or $\TT^3$ such that $\Phi(\cS)$ is a union of vortex lines
and vortex tubes of~$u$. Furthermore, this set is structurally stable.
\end{theorem}

An important observation is that the proof of this theorem yields a
reasonably complete understanding of the behavior of the
diffeomorphism~$\Phi$, which is, in particular, connected with the identity. Oversimplifying a little, the
effect of $\Phi$ is to uniformly rescale a contractible subset of the
manifold that contains~$\cS$ to have a diameter of order
$1/|\la|$. In particular, the control that we have over the
diffeomorphism~$\Phi$ allows us to prove an analog of this result for
quotients of the sphere by finite groups of isometries (lens spaces). Notice that $\Phi(\cS)$ is not guaranteed to contain all vortex
lines and vortex tubes of the Beltrami field. It is also worth
mentioning that, if $\cS$ only consists of curves, the condition that
the perturbation of the Beltrami field be volume-preserving is not
necessary for the structural stability of~$\Phi(\cS)$, and the smallness
in~$C^4$ can be replaced by a $C^1$~condition.

In $\SS^3$ and $\TT^3$, Theorem~\ref{T.main} proves a conjecture of
Arnold~\cite{Ar65} asserting that there should be Beltrami fields
having stream lines with complicated topology. Furthermore, it should
be noticed that the helicity $\cH(\curl u)$ of the vorticity is
proportional to its eigenvalue $\la$ so the Beltrami fields
constructed in the main theorem have very large helicity. More
precisely, the scale-invariant quantity $\cH(\curl u)/\|u\|_{L^2}^2$, which
is given by~$\la$ in the case of a Beltrami field, becomes arbitrarily
large. This is fully consistent with Moffatt's
interpretation~\cite{Mo14} of helicity as a measure of the degree of
knottedness of the vortex lines in the fluid flow.

The proof of the theorem involves an interplay between rigid and
flexible properties of high-energy Beltrami fields. Indeed, rigidity
appears because high-energy Beltrami fields in any 3-manifold behave,
locally in sets of diameter $1/\la$, as Beltrami fields in $\RR^3$
with parameter $\la=1$ do in balls of diameter~1. The catch here is
that, in general, one cannot check whether a given Beltrami field in
$\RR^3$ actually corresponds to a high-energy Beltrami field on the
compact manifold. To prove a partial converse implication in this
direction (Theorem~\ref{T.approx}), it is key to exploit some flexibility that arises in the
problem as a consequence of the fact that large eigenvalues of the
curl operator in the torus or in the sphere have increasingly high
multiplicities. For this reason the proof does not work in a general
Riemannian 3-manifold.

One should notice that the techniques introduced in~\cite{Annals,Acta}
to prove the existence of Beltrami fields in $\RR^3$ with a prescribed
set $\cS$ of closed vortex lines and vortex tubes do not work for
compact manifolds. The reason is that the proof is based on the
construction of a local Beltrami field in a neighborhood of $\cS$,
which is then approximated by a global Beltrami field in $\RR^3$ using
a Runge-type global approximation theorem. For compact manifolds the
complement of the set $\cS$ is precompact, so we cannot apply the
global approximation theorem obtained in~\cite{Annals,Acta}. In fact,
as is well known, this is not just a technical issue, but a
fundamental obstruction in any approximation theorem of this
sort. This invalidates the whole strategy followed
in~\cite{Annals,Acta} and makes it apparent that new tools are needed
to prove the existence of Beltrami fields with geometrically complex
vortex lines and vortex tubes in compact manifolds.

The paper is organized as follows. In Section~\ref{S.main} we will
prove the main theorem assuming that Theorem~\ref{T.approx}
holds. Theorem~\ref{T.approx} will be proved in Section~\ref{S.sphere}
in the case of the sphere, with the proof of some technical results
relegated to Sections~\ref{S.Prop1}--\ref{S.Prop3}, and in
Section~\ref{S.torus} in the case of the torus. The paper concludes
with some remarks that we present in Section~\ref{S:final}, where in
particular we prove an analog of the main theorem for lens spaces.

\section{Proof of the main theorem}
\label{S.main}
For the ease of notation, we shall write $\MM$ to denote
either~$\TT^3$ (the standard flat 3-torus, $(\RR/2\pi\ZZ)^3$)
or~$\SS^3$ (the unit sphere in $\RR^4$). A Beltrami field $u$ in $\MM$ is an eigenfield of the curl operator, which satisfies
$$
\curl u=\la u\,,
$$
for some nonzero constant $\la$. It is well-known that the spectrum of
$\curl$ in the sphere are the integers of absolute value greater than
or equal to~2, while in $\TT^3$ consists of the real
numbers of the form
$$
\la=\pm|k|
$$
for some $k\in\ZZ^3$. In particular, the spectrum of curl in $\TT^3$
contains the set of integers. Here and in what follows, $|\cdot|$
denotes the usual Euclidean norm of a vector.

The following theorem, whose proof is presented in
Section~\ref{S.sphere}, shows that a Beltrami field $v$ in $\RR^3$ can
be approximated, up to a suitable rescaling, by a
high-energy Beltrami field $u$ in $\MM$. This fact is key to the proof
of Theorem~\ref{T.main} as it implies that the dynamics of any
Beltrami field of~$\RR^3$ in
compact sets can be reproduced in a small ball of $\MM$ by a
high-energy Beltrami field on the manifold, provided that the dynamical properties under
consideration are robust under suitably small perturbations. For
concreteness, we will henceforth assume that~$\la$ is positive; the
case of negative~$\la$ is completely analogous.

For the precise statement of the theorem, let us fix an arbitrary
point $p_0\in\MM$ and take a patch of normal geodesic coordinates
$\Psi:\BB\to B$ centered at $p_0$. Here and in what follows, $B_\rho$ (resp.~$\BB_\rho$) denotes the
ball in~$\RR^3$ (resp.\ the geodesic ball
in~$\MM$) centered at the origin (resp.\ at $p_0$) and of 
radius~$\rho$, and we shall drop the subscript when $\rho=1$. The
theorem will be then stated in terms of the vector field $\Psi_*u$ on
$B$, which is just the expression of the Beltrami field~$u$ in local
normal coordinates. If $u^k(x)$ are the three components of $\Psi_*u$ in the
Cartesian basis $\{e_i\}_{i=1}^3$  of~$\RR^3$, i.e.,
\[
\Psi_*u(x) =\sum_{i=1}^3 u^i(x)\, e_i\,,
\]
we will make use of the rescaled vector field
$$
\Psi_*u\Big(\frac{\cdot}{\la}\Big):=\sum_{i=1}^3
u^i\Big(\frac{\cdot}{\la}\Big)\, e_i\,.
$$

\begin{theorem}\label{T.approx}
Let $v$ be a Beltrami field in $\RR^3$, satisfying
$\curl v=v$. Let us fix
any positive numbers $\ep$ and $m$. Then for any large enough odd integer $\la$ there is a
Beltrami field~$u$, satisfying $\curl u=\la u$ in $\MM$, such that 
\begin{equation}\label{nose}
\bigg\|\Psi_* u\Big(\frac{\cdot}{ \la}\Big)-v\bigg\|_{C^m(B)}<\ep\,.
\end{equation}
\end{theorem}

Let us now show how this result can be exploited to prove the main
theorem. For this, let $\Phi'$ be a diffeomorphism of $\MM$ mapping
the set $\cS$ into the ball $\BB_{1/\la}$, and the ball $\BB_{1/\la}$
into itself. (In $\SS^3$, the existence of such a diffeomorphism is
trivial, while in the case of $\TT^3$ it follows from the assumption
that $\cS$ is contained in a contractible set.) Furthermore, given a
positive number~$\La$ let us
denote the rescaling with factor~$\La$ by $\Theta_\La(x):=\La x$. We can now define a set $\cS'$
of finitely many closed curves and tubes in the ball $B$ as
$$
\cS':=(\Theta_\la\circ \Psi\circ \Phi')(\cS)\,.
$$
The following result is a straightforward consequence of the main theorem in~\cite{Acta}:

\begin{theorem}\label{Acta}
  There is a Beltrami field $v$ in $\RR^3$ satisfying $\curl v=v$ 
  and an orientation-preserving diffeomorphism $\Phi_0$ of $\RR^3$,
  which coincides with the identity in the complement of $B$, such
  that $\Phi_0(\cS')$ is a union of vortex lines and vortex tubes
  of~$v$. Furthermore, this set is structurally stable.
\end{theorem}
\begin{proof}
It was shown in~\cite{Acta} that there is a Beltrami field $\tilde v$
in $\RR^3$, satisfying 
\[
\curl \tilde v=\tilde\la \tilde v
\]
for some small
positive constant~$\tilde\la<1$, and an orientation-preserving diffeomorphism $\widetilde\Phi$ of
$\RR^3$ that is the identity in the complement of $B$ such that
$\widetilde\Phi(\cS')$ is a set of closed vortex lines and vortex
tubes of $\tilde v$. The closed vortex lines are elliptic trajectories
of $\tilde v$ and the boundaries of the vortex tubes are
KAM-nondegenerate invariant tori of $\tilde v$. The theorem follows
setting $v(x):=\tilde v(x/\tilde\la)$, which satisfies the equation $\curl
v=v$ in $\RR^3$, and noticing that $(\Theta_{\tilde\la}\circ
\widetilde\Phi)(\cS')$ is a set of closed vortex lines and vortex
tubes of $v$. Since this set is contained in $B$ because $\tilde\la<1$, it
is standard that there exists a diffeomorphism $\Phi_0$ of $\RR^3$ mapping $\cS'$ onto $\Theta_{\tilde\la}\circ \widetilde\Phi(\cS')$ which is the identity in the complement of $B$. The closed vortex lines in the set $\Phi_0(\cS')$ are structurally stable under $C^1$-small perturbations by the elliptic permanence theorem, while the vortex tubes are structurally stable under $C^4$-small volume-preserving perturbations by the KAM theorem.  
\end{proof}

Let us now combine Theorems~\ref{T.approx} and~\ref{Acta} to conclude
the proof of Theorem~\ref{T.main}.
Theorem~\ref{T.approx} guarantees that, for any large enough odd
integer $\la$, the Beltrami field~$v$ constructed in
Theorem~\ref{Acta} can be approximated  in the sense of
Eq.~\eqref{nose} by a Beltrami field~$u$ defined on~$\MM$. Then it is
not hard to see that the
structural stability of the set $\Phi_0(\cS')$ of closed vortex lines
and vortex tubes of $v$ implies the existence of a diffeomorphism
$\Phi_1$ of~$\RR^3$, which is the identity in the complement of $B$,
such that $\Phi_1(\cS')\subset B$ is a set of structurally
stable closed vortex lines and vortex tubes of the rescaled field
\begin{equation}\label{eqqq}
\Psi_* u\Big(\frac{\cdot}{ \la}\Big)\,.
\end{equation}
Indeed, because of the elliptic permanence theorem, this claim is
immediate in the case of closed vortex lines provided that the number
$m$ appearing in the approximation estimate~\eqref{nose} is at
least~1. For the case of vortex tubes one can use that the Beltrami
field~$u$ is divergence-free in $\MM$, which ensures that the field~\eqref{eqqq}
preserves a smooth volume 3-form in $B$ that is a small perturbation of the
Euclidean one, namely
$$
(\Phi_*\mu)\Big(\frac{\cdot}{\la}\Big)=\mu_0+O(\la^{-1})\,.
$$
Here $\mu$ and $\mu_0$ respectively denote the canonical volume
3-forms of $\MM$ and $\RR^3$. Hence, taking $m\geq 4$ in the
approximation estimate~\eqref{nose}, this enables us to apply the KAM
theorem for volume-preserving fields in $\RR^3$, which ensures the
existence of the aforementioned diffeomorphism~$\Phi_1$ yielding the
desired set of vortex tubes of the rescaled field~\eqref{eqqq}. (For
the benefit of the reader let us recall that, in order to prove this
KAM result, one takes a Poincar\'e section transversal to the tube
of~$v$ under consideration, thereby reducing the problem to
perturbations of a nondegenerate twist map of the annulus with the
intersection property. It is then standard that one can apply a
Moser-type twist theorem to guarantee the preservation of the
invariant tori. The details, which go as in~\cite[Section 7.4]{Acta}, are omitted.)

It follows from the above discussion that the diffeomorphism $\Phi$ of
$\MM$ can be then defined as
$$
\Phi(x):=\begin{cases}
\Phi'(x)& \text{if } x\not\in \Phi'^{-1}(\BB_{1/\la})\,,\\
(\Psi^{-1}\circ \Theta_{1/\la}\circ \Phi_1\circ\Theta_\la\circ\Psi\circ\Phi' )(x)& \text{if }x\in \Phi'^{-1}(\BB_{1/\la})\,.
\end{cases}
$$
The set $\Phi(\cS)$ is then the union of structurally stable closed
vortex lines and vortex tubes of the Beltrami field $u$, so the main
theorem follows.

\section{Proof of Theorem~\ref{T.approx} in the sphere}
\label{S.sphere}

In this section we show that for any Beltrami field $v$ in $\RR^3$
there exists a Beltrami field $u$ in $\SS^3$ satisfying
$\curl u=\la u$ whose dynamics in a ball of radius $\la^{-1}$ is very
close to the dynamics of $v$ in the unit ball. The proof is divided in
three steps. In the first step we show that the Beltrami field $v$ can
be approximated in $B$ by a field $w$ that is a finite sum of
spherical Bessel functions $j_0(|x-x_n|)$ centered at different points
$x_n\in\RR^3$ (Proposition~\ref{P.Bessel}). The field $w$ is not
generally a Beltrami field, however. In the second step we show that
one can take three spherical harmonics $Y_1,Y_2,Y_3$ in $\SS^3$ of
energy $\la(\la-2)$ whose behaviors in a ball of radius $1/\la$
respectively correspond to those of the three components of the field
$w$ in a ball of radius $1$, provided that $\la$ is large enough
(Proposition~\ref{P.sphharm}). Finally, in the third step we construct
a Beltrami field $u$ in $\SS^3$ of energy $\la$, using as key
ingredients the spherical harmonics $Y_k$ and a basis of Hopf fields,
so that $u$ approximates the field $v$ in the sense of
Eq.~\eqref{nose} (Proposition~\ref{P.Hopf}).

For notational convenience, in this section we will write
$\La:=\la-2$. Notice that $\La$ is then a large integer.

\subsubsection*{Step 1: Approximating the Beltrami field~$v$ by sums of
  shifted spherical Bessel functions}

The first step of the proof of Theorem~\ref{T.approx} consists in
showing that there is a finite sum $w$ of spherical Bessel functions~$j_0$ centered at different
points that approximates the Beltrami field $v$ in the unit ball of
$\RR^3$. The field~$w$ is not a Beltrami field but, just as~$v$, it
satisfies the Helmholtz equation
$$
\Delta w+w=0\,.
$$

\begin{proposition}\label{P.Bessel}
For any $\de>0$, there is a finite radius $R$ and finitely many
constants $\{c_n\}_{n=1}^N\subset\RR^3$ and $\{x_n\}_{n=1}^N\subset B_R$ such that the
field
\[
w:=\sum_{n=1}^N c_n\, j_0(|x-x_n|)
\]
approximates the Beltrami field~$v$ in the ball $B$ as
\[
\|v-w\|_{C^{m+2}(B)}<\de\,.
\]
\end{proposition}

The proof of this proposition will be presented in Section~\ref{S.Prop1}.

\subsubsection*{Step 2: Approximating the field~$w$ by high-energy
  spherical harmonics}
Let us write the vector field $w$ in terms of its components $w^i$ in the Cartesian basis $\{e_i\}_{i=1}^3$ of~$\RR^3$:
$$
w=\sum_{i=1}^3w^ie_i\,.
$$
Each component $w^i$ is a solution of the Helmholtz equation $\Delta
w^i+w^i=0$ in $\RR^3$. We now show that for any large enough integer
$\La$, there exists a spherical harmonic $Y_i$ on $\SS^3$ with
energy~$\La(\La+2)$ that behaves in the ball $\BB_{1/\La}$ as $w^i$
does in the unit ball. The proof of this result is based on the
asymptotic expressions for the ultraspherical polynomials, which are
the building blocks for any spherical harmonic on $\SS^3$, and
exploits in a crucial way the expression for~$w$ as a finite sum of
spherical Bessel functions that we obtained in Step~1:

\begin{proposition}\label{P.sphharm}
Given any positive constant $\de$, for any large enough integer~$\La$ there is a spherical harmonic $Y_i$ on
$\SS^3$ with energy~$\La(\La+2)$ such that
\begin{equation*}
\bigg\|w^i-Y_i\circ \Psi^{-1}\Big(\frac\cdot \La\Big)\bigg\|_{C^{m+2}(B)}<\de\,.
\end{equation*}
\end{proposition}

The proof of this proposition is given in Section~\ref{S.Prop2}.

\subsubsection*{Step 3: Construction of the Beltrami field on~$\SS^3$
  using spherical harmonics and Hopf fields}

Let us consider the three positively oriented orthonormal Hopf
vector fields in $\SS^3$ that, in terms of the Cartesian coordinates
of~~$\RR^4$, are explicitly given by
\begin{align*}
h_1&:=(-x_4,x_3,-x_2,x_1)\,,\\
h_2&:=(-x_3,-x_4,x_1,x_2)\,,\\
h_3&:=(-x_2,x_1,x_4,-x_3)\,.
\end{align*}
It is well known that they are curl
eigenfields with eigenvalue $2$, that is,
$$
\curl h_i=2h_i\,.
$$
We have taken the the Cartesian basis $e_i$ of~$\RR^3$ so that $\Psi_*h_i(0)=e_i$.

In the following proposition we show how to construct a Beltrami field on $\SS^3$ using the spherical harmonics $Y_i$ obtained in Proposition~\ref{P.sphharm} and the Hopf fields $h_i$ so that it approximates the Beltrami field $v$ in a suitable sense. 

\begin{proposition}\label{P.Hopf}
The vector field on the sphere
\[
u:=\frac1{2\La^2}\curl(\curl+\La)\,(Y_1h_1+Y_2h_2+Y_3h_3)
\]
is a Beltrami field satisfying $\curl u=(\La+2)u$ and approximates~$v$ as
\[
\bigg\| \Psi_*u\Big(\frac\cdot \La\Big)-v\bigg\|_{C^m(B)}<C\de\,,
\]
provided that $\La$ is sufficiently large. 
\end{proposition}

Here $C$ is a constant depending on $m$ but not on $\de$. Since
rescaling $\Psi_*u$ by $\La$ is essentially equivalent to rescaling it
by~$\la$ because
\[
\frac1\La=\frac1\la\,\bigg(1 + \frac2\La\bigg)\,,
\]
Theorem~\ref{T.approx} then follows from Proposition~\ref{P.Hopf}
provided $\La$ is sufficiently large and $\de$ is
chosen small enough for $C\de$ not to be larger than~$\ep/2$. The proof
of Proposition~\ref{P.Hopf} is given in Section~\ref{S.Prop3}.

\section{Proof of Proposition~\ref{P.Bessel}}
\label{S.Prop1}

Since the Beltrami field $v$ satisfies the Helmholtz equation $\Delta
v+v=0$, upon expanding the components of~$v$ in a series of spherical
harmonics it is elementary to realize that~$v$ can be written in the
ball $B_2$ as a Fourier--Bessel series of the form
\begin{equation}\label{fbser}
v=\sum_{l=0}^\infty\sum_{m=-l}^l b_{lm}\,j_{l}(r)\, Y_{lm}(\om)
\end{equation}
that converges in $L^2(B_2)$.
Here $r:=|x|\in\RR^+$ and $\om:=x/r\in\SS^2$ are spherical
coordinates, $j_l$ is the spherical Bessel function, $Y_{lm}$ are the spherical harmonics and $b_{lm}\in\RR^3$ are constant vectors.

Since the series~\eqref{fbser} converges in $L^2(B_2)$, for any $\de'$ there is an integer $l_0$ such that the finite sum
\begin{equation*}
v_1:=\sum_{l=0}^{l_0}\sum_{m=-l}^l b_{lm}\, j_{l}(r)\, Y_{lm}(\om)
\end{equation*}
approximates the field $v$ in an $L^2$ sense, that is,
\begin{equation}\label{L2b}
\|v_1-v\|_{L^2(B_2)}<\de'\,.
\end{equation}

Next, let us observe that the properties of the spherical Bessel
functions imply that the field $v_1$ falls off at infinity as
$|v_1(x)|<C/|x|$. In particular, it then follows from Herglotz's
theorem (see e.g.~\cite[Theorem 7.1.27]{Hormander}) that $v_1$ can be
written as the Fourier transform of a distribution supported on the unit sphere of the form  
\begin{equation}
v_1(x)=\int_{\SS^2}f_1(\xi)\, e^{ix\cdot\xi}\, d\si(\xi)\,,
\end{equation}
where $d\si$ is the area measure induced on the unit sphere
$\SS^2:=\{\xi\in\RR^3:|\xi|=1\}$ and $f_1$ is an $\RR^3$-valued
function in $L^2(\SS^2)$.

By the density of smooth functions in $L^2(\SS^2)$, we can approximate
$f_1$ by a smooth function $f_2:\SS^2\to\RR^3$ so that their
difference is bounded as
\[
\|f_1-f_2\|_{L^2(\SS^2)}<\de'\,.
\]
Therefore the field
\begin{equation}
v_2(x):=\int_{\SS^2}f_2(\xi)\,e^{ix\cdot\xi}\, d\si(\xi)\,,
\end{equation}
approximates $v_1$ uniformly, as for any $x\in\RR^3$ the
Cauchy--Schwarz inequality yields
\begin{align}
|v_2(x)-v_1(x)|=\bigg|\int_{\SS^2}(f_2(\xi)-f_1(\xi))\,e^{ix\cdot\xi}\,
  d\si(\xi)\bigg|\leq C\|f_2-f_1\|_{L^2(\SS^2)}<C\de'\,. \label{eqv2v1}
\end{align}

Our next objective is to show that for any $\de'$ there is a radius $R>0$ and
finitely many constants $\{c_n\}_{n=1}^N\subset\RR^3$ and
$\{x_n\}_{n=1}^N\subset B_R$ such that the restriction to the unit
sphere of the smooth field in $\RR^3$  
$$
f(\xi):=\sum_{n=1}^N c_n \, e^{-ix_n\xi}
$$
approximates the field $f_2$ in the $C^0$ norm, that is,
\begin{equation}\label{eqqma}
\|f-f_2\|_{C^0(\SS^2)}<\de'\,.
\end{equation}

To prove this claim, we first extend $f_2$ to a smooth vector field $g:\RR^3\to\RR^3$ with compact support. This can be done setting 
$$
g(\xi):=\chi(|\xi|)\, f_2\bigg(\frac{\xi}{|\xi|}\bigg)\,,
$$
with $\chi(s)$ a smooth function which is equal to $1$, say, if
$|s-1|<\frac14$ and vanishes for $|s-1|>\frac12$. Since the Fourier
transform $\hg$ of $g$ is Schwartz, we easily infer that there is a
radius $R$ such that the $L^1$~norm of $\hg$ is essentially contained
in $B_R$ in the sense that
\[
\int_{\RR^3\backslash B_R}|\hg(x)|\, dx<\de'\,.
\]
It then follows that the Fourier integral representation of $g$ can be essentially
truncated to an integral over the ball $B_R$, i.e., one has the uniform bound
\begin{equation}\label{Rlarge}
\sup_{\xi\in\RR^3}\bigg|g(\xi)-\int_{B_R}\hg(x)\, e^{-ix\cdot\xi}\, dx\bigg|<\de'\,.
\end{equation} 

Now, an easy continuity argument allows us to uniformly approximate the integral
$$
\int_{B_R}\hg(x)\,e^{-ix\cdot\xi}\, dx
$$
by a finite sum 
\begin{equation}\label{fcn}
f(\xi):=\sum_{n=1}^N c_n \, e^{-ix_n\cdot \xi}
\end{equation}
with constants $c_n\in\RR^3$ and $x_n\in B_R$ in such a
way that the error introduced in the approximation is bounded by
\begin{equation}\label{eqdisc}
\sup_{\xi\in \SS^2}\bigg|\int_{B_R}\hg(x)\, e^{-ix\cdot\xi}\, dx-f(\xi)\bigg|<\de'\,.
\end{equation}

Indeed, let us cover the ball $B_R$ by finitely many closed sets
$\{U_n\}_{n=1}^N$ with piecewise smooth boundaries and pairwise
disjoint interiors such that the
diameter of each set is at most $\de''$. The
function~$e^{-ix\cdot\xi}\, \hg(x)$ being smooth, it then follows that for each $x,y\in U_n$ 
one has
\[
\sup_{\xi\in\SS^2}\big|\hg(x)\, e^{-ix\cdot\xi}-\hg(y)\, e^{-i y\cdot\xi}|< C\de''\,,
\]
where the constant $C$ depends on~$\hg$ (and therefore on~$\de'$) but
not on $\de''$. It is then straightforward that if $x_n$ is any point
in~$U_n$ and we set $c_n:=\hg(x_n)\,|U_n|$ in~\eqref{fcn}, one has
\begin{align*}
\sup_{\xi\in\SS^2}\bigg|\int_{B_R}\hg(x)\, e^{-ix\cdot\xi}\,
  dx-f(\xi)\bigg|&\leq \sum_{n=1}^N\int_{U_n} \sup_{\xi\in\SS^2}\big|\hg(x)\,\e^{-i
                   x\cdot\xi}-\hg(x_n)\, e^{-ix_n\cdot\xi}\big|\, dx\\
&\leq C\de''\,,
\end{align*}
where again $C$ depends on $\de'$ and $R$ but not on~$\de''$ or~$N$. Hence
one can take~$\de''$ small enough so that $C\de''<\de'$, thereby
proving the estimate~\eqref{eqdisc}.

Putting together the estimates~\eqref{Rlarge} and~\eqref{eqdisc} we infer that
\begin{equation*}
\|f-g\|_{C^0(\SS^2)}<C\de'\,,
\end{equation*}
with a constant independent of~$\de'$. Since the restriction
to~$\SS^2$ of the function~$g$ is precisely~$f_2$, the
estimate~\eqref{eqqma} then follows. 

Finally, if we define the vector field
\begin{equation*}
w(x):=\int_{\SS^2}f(\xi)\, e^{ix\cdot\xi}\, d\si(\xi)=\sum_{n=1}^N
c_{n}\int_{\SS^2}e^{i(x-x_{n})\cdot \xi}\,d\si(\xi)=\sum_{n=1}^N
c_{n}\, j_{0}(|x-x_{n}|)\,,
\end{equation*}
we conclude from Eq.~\eqref{eqqma} that 
\begin{equation*}
\|w-v_2\|_{C^0(\RR^3)}\leq \int_{\SS^2}|f(\xi)-f_2(\xi)|\, d\si(\xi)<C\de'\,,
\end{equation*}
so we readily infer from Eqs.~\eqref{L2b} and~\eqref{eqv2v1} the
$L^2$ bound
\begin{align}\label{cas}
\|v-w\|_{L^2(B_2)}\leq C\|w-v_2\|_{C^0(\RR^3)}+C\|v_2-v_1\|_{C^0(\RR^3)}+\|v_1-v\|_{L^2(B_2)}<C\de'\,.
\end{align}
Furthermore, as the Fourier transform of~$w$ is supported on~$\SS^2$, $w$ satisfies the Helmholtz equation
$$
\Delta w+w=0
$$
in $\RR^3$. Since the Beltrami field $v$ also satisfies the Helmholtz
equation $\Delta v+v=0$, standard elliptic estimates enable us to
promote the $L^2$ bound~\eqref{cas} to the $C^{m+2}$ estimate
\begin{equation*}
 \|v-w\|_{C^{m+2}(B)}\leq C\|v-w\|_{L^2(B_2)}< C\delta'\,,
\end{equation*}
so the proposition follows upon choosing $C\delta'<\de$.

\section{Proof of Proposition~\ref{P.sphharm}}
\label{S.Prop2}

For any positive integer $\La$, let $C_{\La}(t)$ be the ultraspherical (also called Gegenbauer) polynomial of dimension
4 and degree~$\La$, which can be defined in terms of the Jacobi polynomials
$P_{\La}^{(\alpha,\,\beta)}$ as
\begin{equation}\label{cla}
C_\La(t):=\frac{\sqrt{\pi}}{2}\frac{\Gamma(\La+1)}{\Gamma(\La+\frac{3}{2})}\,P_{\La}^{(\frac{1}{2},\,\frac{1}{2})}(t)\,,
\end{equation}
where we are using the normalization $C_\La(1)=1$ for all $\La$.

If $p,q\in \SS^3$ are two points in the 3-sphere, understood as the
subset $\{|p|=1\}$ of~$\RR^4$, the addition theorem for ultraspherical
polynomials shows that $C_\La(p\cdot q)$ can be written as a linear
combination of spherical harmonics. Specifically,
\begin{equation}\label{clasp}
C_{\La}(p\cdot q)=\frac{2\pi^2}{(\La+1)^{2}}\sum_{j=1}^{(\La+1)^{2}}Y_{\La j}(p)\,Y_{\La j}(q)\,,
\end{equation}
where $\{Y_{\La j}\}_{j=1}^{(\La+1)^{2}}$ is an arbitrary orthonormal
basis of spherical harmonics of energy $\La(\La+2)$ and $p\cdot q$
denotes the scalar product in~$\RR^4$ of the unit vectors $p,q$. Notice
that $(\La+1)^2$ is precisely the multiplicity of the eigenvalue
$\La(\La+2)$ of the Laplacian on~$\SS^3$ (or, equivalently, the dimension of the space of spherical harmonics
of energy $\La(\La+2)$).

Let us write the $i^{\text{th}}$ Cartesian component of the vector field $w$ as 
\[
w^i(x)=\sum_{n=1}^Nc_n^i\,j_{0}(|x-x_{n}|)\,,
\]
where $c_n^i$ is the $i^{\text{th}}$ component of the constant
$c_n\in\RR^3$ and the points $x_n$ are contained in the ball $B_R$.
Let us set, for any $p\in\SS^3$,
\begin{equation*}
Y_i(p):=\sum_{n=1}^N c_n^i\, C_\La(p\cdot p_n)\,,
\end{equation*}
with
\[
p_n:=\Psi^{-1}\bigg(\frac{x_{n}}{\La}\bigg)\,.
\]
Note that $p_n$ is well defined provided $\La$ is bigger than~$R$. It is obvious from Eq.~\eqref{clasp} that $Y_i$ is a spherical harmonic of energy $\La(\La+2)$.

In order to study the asymptotic properties of the spherical harmonic
$Y_i$ we first observe that, if we restrict our attention to points of
the form
\[
p:=\Psi^{-1}\bigg(\frac{x}{\La}\bigg)
\]
with $x\in B_R$ and $\La>R$, we then have
\begin{equation}\label{cos}
p\cdot p_n=\cos\big(\text{dist}_{\SS^3}(p,p_n)\big)=\cos \bigg(\frac{|x-x_n|+O(\La^{-1})}{\La}\bigg)\,,
\end{equation}
as $\La\to \infty$. Here $\text{dist}_{\SS^3}(p,p_n)$ denotes the
distance between the points $p$ and $p_n$ as measured on the sphere
$\SS^3$ and the last equality stems from the fact that $\Psi:\BB\to B$
is a patch of normal geodesic coordinates.  We will henceforth use the
notation
\begin{equation}\label{defi}
\tY_i(x):=Y_i\circ\Psi^{-1}\bigg(\frac{x}{\La}\bigg)\,.
\end{equation}

Since for $\La$ large we have the asymptotic behavior
$$
\frac{\Gamma(\La+1)}{\Gamma(\La+\frac{3}{2})}=\frac{\text{1}}{\sqrt{\La}}+O(\La^{-\frac{3}{2}})\,,
$$
we conclude from Eq.~\eqref{cos} that
\begin{equation}\label{aaa}
C_\La(p\cdot
p_n)=\bigg(\frac{\sqrt\pi}{2\sqrt{\La}}+O(\La^{-\frac{3}{2}})\bigg)\,
P_{\La}^{(\frac{1}{2},\,\frac{1}{2})}\bigg(
\cos\bigg(\frac{|x-x_n|+O(\La^{-1})}{\La}\bigg)\bigg)\,.
\end{equation}
Now Darboux's asymptotic formula for the Jacobi polynomials~\cite[Theorem 8.1.1]{Szego75} implies
\begin{equation*}
\frac{1}{\sqrt{\La}}P_{\La}^{(\frac{1}{2},\,\frac{1}{2})}\Big(\cos\frac{t}{\La}\Big)=
\frac{2}{\sqrt{\pi}}\, j_{0}(t)+O(\La^{-1})\,,
\end{equation*}
which holds uniformly for compact sets (e.g., for $|t|\leq
2R$). Therefore Eq.~\eqref{defi} can be written by virtue of Eq.~\eqref{aaa} as
\begin{align*}
\tY_i(x)&=\sum_{n=1}^N c_n^i C_\La\Big(\cos\Big(\frac{|x-x_n|+O(\La^{-1})}{\La}\Big)\Big)\\&=\sum_{n=1}^{N}c_n^i\,j_0(|x-x_n|)+O(\La^{-1})\,,
\end{align*}
provided that $\La$ is sufficiently large and $x, x_n\in B_R$. This
proves that for any $\delta'>0$ and all $\La$ large enough we have the
uniform bound
\begin{equation}\label{c0}
\|w^i-\tY_i\|_{C^0(B)}<\de'\,.
\end{equation}

To get the $C^{m+2}$ bound stated in the proposition, we notice that the eigenvalue equation 
$$\Delta Y_i+\La(\La+2)Y_i=0$$  
for the spherical harmonic $Y_i$ in $\SS^3$ can be written in terms of
the rescaled function~$\tY_i$ as
\begin{equation*}
\Delta_0\tY_i+\tY_i=\frac{1}{\La}A\tY_i\,,
\end{equation*}
where the coordinates $x$ are assumed to take values in $B$,
$\Delta_0:=\sum_i \pd_{x_i}^2$ is the flat space Laplacian acting on the
$x$ coordinates and $A$ is a scalar second-order operator of
the form
$$
A\tY_i:=-2\tY_i+G_1\, D\tY_i+G_2\, D^2\tY_i\,.
$$
Here the functions $G_i(x,\La)$ are (possibly matrix-valued) functions
that depend smoothly on all their variables and whose derivatives are
bounded independently of~$\La$ for $x\in B$, i.e.,
\begin{equation}\label{bg2a}
\sup_{x\in B}|D^\alpha_x G_i(x,\La)|< C_{\al}\,.
\end{equation} 
Here the constant $C_{\al}$ depends on the multiindex~$\al$ but not on $\La$.

By construction, the function $w^i$ satisfies the Helmholtz equation
$$
\De_0w^i+w^i=0\,,
$$
and hence the difference $w^i-\tY_i$ satisfies the equation
$$
\De_0(w^i-\tY_i)+(w^i-\tY_i)=\frac{1}{\La}A\tY_i\,.
$$
Therefore, in view of the uniform bounds~\eqref{c0} and~\eqref{bg2a},
standard elliptic estimates yield
\begin{align*}
\|w^i-\tY_i\|_{C^{m+2,\alpha}(B)}&<C\|w^i-\tY_i\|_{C^0(B)}+\frac{C}{\La}\|A\tY_i\|_{C^{m,\alpha}(B)}\\
&<C\delta'+\frac{C}{\La}\|w^i-\tY_i\|_{C^{m+2,\al}(B)}+\frac{C}{\La}\|w^i\|_{C^{m+2,\alpha}(B)}\,,
\end{align*}
which implies that
$$
\|w^i-\tY_i\|_{C^{m+2}(B)}\leq C\de'+\frac{C \|w_i\|_{C^{m+2,\al}}}{\La}<\de
$$
provided that $\La$ is large enough (which in turn implies that $\de'$
is small). This completes the proof of the proposition.

%
%

\section{Proof of Proposition~\ref{P.Hopf}}
\label{S.Prop3}

We start by defining a vector field $\tu$ on $\SS^3$ using the Hopf fields $h_i$ as
$$\tu:=Y_{1} \,h_{1}+Y_{2}\,h_{2}+Y_{3}\,h_{3}\,,$$
where the functions $Y_i$ are the spherical harmonics obtained in
Proposition~\ref{P.sphharm}. In what follows it is convenient to work
with differential forms, so let us denote by $\tbe$ and $\al_i$ the
$1$-forms that are dual to $\tu$ and $h_i$, respectively, with respect
to the canonical metric on $\SS^3$. We recall that the dual of $\curl \tu$ is the $1$-form $\star d\tbe$, with $\star$~being the Hodge star operator. 

In the following lemma we compute the action of the Hodge Laplacian on the 1-form $\tbe$ using the properties of the Hopf fields:

\begin{lemma}
The Hodge Laplacian of the $1$-form $\tbe$ dual to $\tu$ is
$$
-\De\tbe=\La(\La+2)\,\tbe+2\star d\tbe\,,
$$
\end{lemma}
\begin{proof}
The $1$-form $\tbe$ is given by $\tbe=Y_i\, \alpha_i$, where summation
over repeated indices is understood throughout. The Laplacian of $\tbe$ is then 
$$-\De\tbe:=dd^{\ast}\tbe+d^{\ast}d\tbe=-d\star d\star(Y_i\, \alpha_i)+\star d\star d(Y_i\, \alpha_i)\,.$$

Using that $\star d\alpha_i=2\alpha_i$ because $\alpha_i$ is the dual
$1$-form of the Hopf field $h_i$, and that the differential of $Y_i$
can be written as $dY_i=h_j(Y_i)\, \alpha_j$, where $h_j(Y_k)$ denotes
the action of the vector field $h_j$ on the scalar function $Y_k$, we
readily obtain 
$$d\star d\star(Y_i\, \alpha_i)=\frac12 d\star(h_{j}(Y_i)\, \alpha_{j}\wedge d\alpha_i)\,.$$ 
Observing that $\alpha_{j}\wedge
d\alpha_i=2\alpha_{j}\wedge\star\alpha_i=2\delta_{jk}\, \mu$, where
$\mu$ stands for the Riemannian volume 3-form on $\SS^3$, it follows that
\begin{equation}\label{eql1}
d\star d\star(Y_i\, \alpha_i)=d(h_i(Y_i))=h_{j}h_i(Y_i)\, \alpha_{j}\,.
\end{equation}

Analogously, a straightforward computation using that
$\star(\al_j\wedge\al_i)=\varepsilon_{jil}\al_l$, where
$\varepsilon_{jil}$ stands for the Levi-Civita permutation symbol, and the identity $\varepsilon_{iml}\varepsilon_{jkl}=\delta_{ij}\delta_{mk}-\delta_{ik}\delta_{mj}$ yields
\begin{align}
&\star d(Y_i\, \al_i)=\varepsilon_{jil}h_j(Y_i)\, \al_l+2Y_i\, \al_i\,,\label{eql2}\\
&\star d\star d(Y_i\, \alpha_i)=-h_jh_j(Y_i)\, \al_i+h_ih_j(Y_i)\, \al_j+4\varepsilon_{jil}h_j(Y_i)\, \al_l+
4Y_i\, \al_i\,.\label{eql3}
\end{align}
Finally, adding Eqs.~\eqref{eql1} and~\eqref{eql3} we obtain
\begin{align*}
-\De\tbe&=-h_{j}h_i(Y_i)\, \alpha_{j}+h_ih_{j}(Y_i)\, \alpha_{j}-h_{j}h_{j}(Y_i)\, \alpha_i+
4\varepsilon_{jil}h_j(Y_i)\, \al_l+4Y_i\, \al_i\\
&=\La(\La+2)Y_i\, \al_i+2\varepsilon_{jil}h_j(Y_i)\, \al_l+4Y_i\, \al_i\,,
\end{align*}
where we have used that $\Delta Y_i=-\La(\La+2)Y_i$ and that the
commutator of Hopf fields is $[h_i,h_j]=-2\varepsilon_{ijl}h_l$. The
lemma then follows upon noticing that 
$$
2\varepsilon_{jil}h_j(Y_i)\, \al_l+4Y_i\, \al_i=2\star d\tbe
$$ 
by Eq.~\eqref{eql2}.
\end{proof}

Using this lemma, it is easy to check that 
$$u:=\frac{1}{2\La^2}\curl(\curl+\La)\tu$$ 
is a Beltrami field with eigenvalue $\La+2$. Indeed, if $\beta$ is the dual $1$-form of $u$, we obtain
\begin{align*}
\star d\beta=\frac{1}{2\La^2}\star d\star d(\star d+\La)\tbe&=\frac{1}{2\La^2}\star d(-\De+\La\star d)\tbe\\&=\frac{\La+2}{2\La^2}\star d(\star d+\La)\tbe=(\La+2)\beta\,.
\end{align*}

To prove the $C^m$ estimate of the proposition, it is convenient to
introduce an auxiliary vector field in the unit ball $B$ of $\RR^3$ as
$$\Bu(x):=\tY_{1}(x)\,
e_1+\tY_{2}(x)\, e_2+
\tY_{3}(x)\,e_3\,,$$
where $x\in B$ and $\tY_i$ was defined in~\eqref{defi}. There is no
loss of generality in choosing the orthonormal basis $e_i$ of
$\mathbb{R}^{3}$ compatible with the Hopf fields $h_i$ in the sense
that $\Psi_{\ast}(h_i)(0)=e_i$. It is then easy to check that for
$x\in B$ one has:
\begin{align*}
\Psi_*\tu\Big(\frac{\cdot}{\La}\Big)&= \Bu+\frac{G_1}{\La}\Bu\,,\\
\Psi_*(\curl\tu)\Big(\frac{\cdot}{\La}\Big)&= \La\Big(\curl_0\Bu+\frac{G_2}{\La}\Bu+\frac{G_3}{\La}D\Bu\Big)\,,\\
\Psi_*(\curl\curl\tu)\Big(\frac{\cdot}{\La}\Big)&= \La^2\Big(\curl_0\curl_0\Bu+\frac{G_4}{\La}\Bu+\frac{G_5}{\La}D\Bu+\frac{G_6}{\La}D^2\Bu\Big)\,.\\
\end{align*}
Here $\curl_0$ denotes the Euclidean curl operator, acting on the
variables~$x$, and the functions $G_i(x,\La)$ are (possibly
matrix-valued) functions that depend smoothly on all their variables
and whose derivatives are uniformly bounded as
\begin{equation}\label{bg2}
\sup_{x\in B}|D^\alpha_x G_i(x,\La)|< C_\al\,.
\end{equation} 
Here the constant $C_\al$ depends on the multiindex~$\al$ but not on
$\La$. 

These identities and the fact that $(\curl_0\curl_0+\curl_0)v=2v$
then permits us to write
\begin{align}\label{finalest}
\Big\|\Psi_*u\Big(\frac{\cdot}{\La}\Big)-v\Big\|_{C^m(B)}&\leq\Big\|\frac12(\curl_0\curl_0
                                                           +\curl_0)(\Bu-v)\Big\|_{C^m(B)}+\frac{C}{\La}\|\Bu\|_{C^{m+2}(B)}\notag
  \\&\leq
                                                                                                                                C\|\Bu-v\|_{C^{m+2}(B)}+\frac{C}{\La}\|\Bu-w\|_{C^{m+2}(B)}\notag\\
&\qquad \qquad \qquad \qquad \quad+\frac{C}{\La}\|v-w\|_{C^{m+2}(B)}+\frac{C}{\La}\|v\|_{C^{m+2}(B)}\,.  
\end{align}
To conclude, notice that it stems from Propositions~\ref{P.Bessel} and~\ref{P.sphharm} that
\begin{align*}
\|v-w\|_{C^{m+2}(B)}&<\de\\
\|\Bu-w\|_{C^{m+2}(B)}&< 3\delta\,,  
\end{align*}
so in particular
\[
\|\Bu-v\|_{C^{m+2}(B)}\leq \|\Bu-w\|_{C^{m+2}(B)}+\|v-w\|_{C^{m+2}(B)}< 4\delta\,.
\]
Hence the proposition follows from the estimate~\eqref{finalest} upon
noticing that $v$ is a fixed vector field (so its norm is independent
of~$\La$) and choosing $\La$ large enough, which also allows us to take
$\de$ as small as one wishes.

\section{Proof of Theorem~\ref{T.approx} in the torus}
\label{S.torus}

Arguing as in the proof of Proposition~\ref{P.Bessel} we can readily
show that for any $\de>0$, there exists a vector field $v_1$ on $\RR^3$ that approximates the Beltrami field $v$ in the ball $B$ as
\begin{equation}\label{otrv}
\|v_1-v\|_{C^{0}(B)}<\de\,,
\end{equation}
and that can be represented as the Fourier transform of a distribution supported on the unit sphere of the form  
\begin{equation*}
v_1(x)=\int_{\SS^2}f(\xi)\, e^{i\xi\cdot x}\, d\si(\xi)\,.
\end{equation*}
Again $\SS^2$ denotes the unit sphere $\{\xi\in\RR^3:|\xi|=1\}$ and
$f$ is a smooth $\RR^3$-valued function on $\SS^2$.

Let us now cover the sphere $\SS^2$ by finitely many closed sets
$\{U_n\}_{n=1}^N$ with piecewise smooth boundaries and pairwise
disjoint interiors such that the
diameter of each set is at most $\de'$. We can then repeat the argument
used in the proof of Proposition~\ref{P.Bessel} to infer that, if
$\xi_n$ is any point in $U_n$ and we set
\[
c_n:=f(\xi_n)\, |U_n|\,,
\]
the field 
\[
w(x):=\sum_{n=1}^N c_{n}\, e^{i\xi_{n}\cdot x}
\]
approximates the field $v_1$ uniformly with an error proportional to~$\de'$:
\begin{equation*}
\|w-v_1\|_{C^{0}(B)}<C\delta'\,.
\end{equation*}
The constant $C$ depends on $\de$ but not on $\de'$, so one can choose
the maximal diameter~$\de'$ small enough so that
\begin{equation}\label{mases}
\|w-v_1\|_{C^{0}(B)}<\delta\,.
\end{equation}
In turn, the uniform estimate 
\[
\|w-v\|_{C^0(B)}\leq \|w-v_1\|_{C^0(B)}+ \|v-v_1\|_{C^0(B)}<2\de
\]
can be readily promoted to the $C^{m+2}$ bound
\begin{equation}\label{estder}
\|w-v\|_{C^{m+2}(B)}<C\de\,.
\end{equation}
This follows from standard elliptic estimates as both $w$ (whose
Fourier transform is supported on~$\SS^2$) and $v$ satisfy the
Helmholtz equation:
\[
\De v +v =0\,,\qquad \De w+ w=0\,.
\]
Furthermore, replacing $w$ by its real part if necessary, we can safely
assume that the field~$w$ is real-valued.

Let us now observe that for any large enough odd integer~$\La$ one can choose the points $\xi_n\in U_n\subset\SS^2$ so
that they have rational components (i.e., $\xi_{n}\in
\mathbb{Q}^{3}$) and the rescalings $\La\xi_n$ are actually integer
vectors (i.e., $\La \xi_n\in\ZZ^3$). This is because rational points
$\xi\in\SS^2\cap \mathbb{Q}^3$ with $\La\xi\in\ZZ^3$ are uniformly distributed on the unit sphere as $\La\to \infty$ through odd values~\cite{Du03}.

Choosing $\xi_n$ as above, we are now ready to prove
Theorem~\ref{T.approx} in the torus. Without loss of generality, we
will take the origin as the base point $p$, so that we can identify
the ball $\BB$ with $B$ through the canonical $2\pi$-periodic
coordinates on the torus. In particular, the
diffeomorphism~$\Psi:\BB\to B$
that appears in the statement of Theorem~\ref{T.approx} can be
understood to be the identity.

Since $\La
\xi_n\in\ZZ^3$, it follows that the vector field
\begin{equation*}
\tu(x):=\sum_{n=1}^Nc_{n}e^{i\La\xi_n\cdot x}
\end{equation*}
is $2\pi$-periodic (that is, invariant under the translation $x\to
x+2\pi\, a$ for any vector $a\in\ZZ^3$). Therefore it descends to a well-defined vector field on the torus $\TT^3:=\RR^3/(2\pi\ZZ)^3$, which we will still denote by $\tu$. 

Since the Fourier transform of $\tu$ if now supported on the sphere of
radius~$\La$, $\tu$ then satisfies the Helmholtz equation on the flat torus $\TT^3$ with energy $\La^2$,
$$
\Delta\tu+\La^{2}\tu=0\,.
$$
A straightforward calculation then reveals that the vector field on
the torus
\begin{equation*}
u:=\frac{\curl\curl \tu+\La\curl \tu}{2\La^2}
\end{equation*}
satisfies the equation 
$$
\curl u=\La u\,,
$$
so it is a Beltrami field on $\TT^3$ with eigenvalue
$\la:=\La$.

Let us now notice that, with some abuse of notation,
\[
\tu\bigg(\frac x\La\bigg)=w(x)
\]
for all points $x$, say, in the ball $B$. In particular, as
the derivatives of the rescaled vector field $\tu(\cdot/\La)$ behave as
\begin{align*}
\curl \tu\Big(\frac{\cdot}{\La}\Big)=& \La\curl w\,,\\
\curl\curl \tu\Big(\frac{\cdot}{\La}\Big)=& \La^2\curl\curl w\,,
\end{align*}
it then follows that
\begin{align*}
\bigg\|u\Big(\frac{\cdot}{\La}\Big)-v\bigg\|_{C^{m}(B)}&
                                                         =\bigg\|\frac{\La^{2}\curl\curl w+\La^{2}\curl w}{2\La^{2}}-v\bigg\|_{C^{m}(B)}\\[1mm]& =\bigg\|\frac{\curl\curl (w-v)+\curl (w-v)}{2}\bigg\|_{C^{m}(B)}\\[1mm] & \leq C\|w-v\|_{C^{m+2}(B)}\\
&<C\de\,,
\end{align*}
where we have used the identity $\curl\curl v+\curl v=2v$ to pass to
the second equality and the estimate~\eqref{estder} to derive the last
inequality. The theorem then follows provided that $\de$ is chosen
small enough for $C\de<\ep$.

\section{Concluding remarks}\label{S:final}

To conclude, let us make a few simple observations about our main
result that follow from its proof:

\subsubsection*{There are many Beltrami fields with closed vortex lines and
  tubes of a given link type} Indeed, since our construction works for
any large enough odd integer~$\la$ and Beltrami fields corresponding
to different eigenvalues are $L^2$~orthogonal, there are many
non-proportional Beltrami fields with closed vortex lines and tubes
realizing any given link.

\subsubsection*{In the sphere, the result holds true for any large
  enough eigenvalue~$\la$}

Indeed, the fact that $\La$ is odd was
never used in the proof of Theorem~\ref{T.approx} in~$\SS^3$ (cf.\
Section~\ref{S.sphere}), so it stems that, given any finite union of
closed curves and tubes~$\cS$, for any integer~$\la$ with $|\la|$
greater than certain constant $\La_0(\cS)$ there is a Beltrami field
with eigenvalue~$\la$ having a structurally stable set of vortex lines and vortex tubes
diffeomorphic to~$\cS$.

\subsubsection*{In our Beltrami fields on the sphere, knots and links
    appear in pairs} In fact, using the Hopf basis $\{h_i\}_{i=1}^3$  introduced in Section~\ref{S.sphere}, any
  Beltrami field $u$ on $\SS^3$ with eigenvalue $\la:=\La+2$, with $\La$ a nonnegative integer, can be written as
$$
u=F_1\, h_1+F_2\, h_2+F_3\, h_3\,,
$$
where $F_i$ are smooth functions on the sphere. It is then easy to
check using Eq.~\eqref{eql2} that $F_i$ must be a spherical harmonic
of energy $\La(\La+2)$. Since such a spherical harmonic is known to
have parity $(-1)^\La$, in the sense that
\[
F_i(-p)=(-1)^\La \, F_i(p)
\]
for all points $p$ in the unit sphere $\SS^3$, and the Hopf fields
$h_i$ are odd (i.e., $h_i(-p)=-h_i(p)$), we conclude that a Beltrami
field on the sphere with eigenvalue $\la$ has parity $(-1)^{\la+1}$, so it
is either even or odd.
Therefore, the fact that $\Phi(\cS)$  is a set of vortex lines and vortex tubes
of the Beltrami field~$u$ diffeomorphic to $\cS$ and contained in a
ball of small radius $1/\la$ automatically implies
that so is the antipodal set $-\Phi(\cS)$.

\subsubsection*{The result carries over to lens spaces} In order to
see why, the key is that in the sphere the statement of
Theorem~\ref{T.approx} can be refined to include localizations around
different points of the sphere. More precisely, let us fix $l$ points
$P_1,\dots, P_l$ in~$\SS^3$, none of which are antipodal to another
(that is, $P_j\neq -P_k$), and denote by $\Psi_j:\BB(P_j,R_0)\to
B_{R_0}$ a patch of normal geodesic coordinates centered at the point
$P_j$. Here $\BB(P_j,R_0)$ denotes the geodesic ball in the sphere of
center $P_j$ and radius 
\[
R_0:=\frac12\min_{j\neq k}\dist_{\SS^3}(P_j,P_k)\,.
\]
The approximation theorem can then be stated as follows:

\begin{theorem}\label{T.approx2}
Let $\{v_j\}_{j=1}^l$ be Beltrami fields in $\RR^3$, satisfying
$\curl v_j=v_j$. Let us fix
any positive numbers $\ep$ and $m$. Then for any large enough integer $\la$ there is a
Beltrami field~$u$, satisfying $\curl u=\la u$ in $\SS^3$, such that 
\begin{equation}\label{nose}
\bigg\|(\Psi_j)_* u\Big(\frac{\cdot}{ \la}\Big)-v_j\bigg\|_{C^m(B)}<\ep
\end{equation}
for all $1\leq j\leq l$.
\end{theorem}

\begin{proof}
Arguing as in Proposition~\ref{P.sphharm} we
infer that for any large enough integer~$\La$ there are spherical
harmonics $\hat Y_{ij}$ of energy $\La(\La+2)$ such that
\[
\bigg\|w^i_j-\hat Y_{ij}\circ \Psi_j^{-1}\Big(\frac\cdot \La\Big)\bigg\|_{C^{m+2}(B)}<\de\,,
\]
where $w_j$ is a vector field of the form
\[
w_j= \sum_{n=1}^N c_{jn}\, j_0(|x-x_{jn}|)
\]
that approximates the Beltrami field $v_j$ in $C^{m+2}(B)$ as in
Proposition~\ref{P.Bessel} and $w_j^i$ ($1\leq i\leq3$) denotes its
$i^{\mathrm{th}}$ Cartesian component. Noticing that the Jacobi
polynomial behaves as
\[
\La^{-\frac12}\, P_\La^{(\frac12,\frac12)}(\cos t)=\frac{O(\La^{-1})}t
\]
uniformly for $\La^{-1}<t<\pi-\La^{-1}$~\cite[Theorem 7.32.2]{Szego75},
it stems that the ultraspherical polynomial $C_\La$ is uniformly
bounded as
\[
|C_\La(p\cdot q)|\leq \frac{C_\rho}\La
\]
for any points $p,q$ in $\SS^3$ such that
\begin{equation*}
\dist_{\SS^3}(p,q)\geq \rho \quad \text{and}\quad \dist_{\SS^3}(p,-q)\geq \rho \,,
\end{equation*}
with a constant $C_\rho$ that only depends on the positive constant
$\rho$. 

Using the formulas of
Section~\ref{S.Prop2} it is now easy to show that for any $j$ and any fixed positive radius $\rho$ we have
\[
\|\hat Y_{ij}\|_{C^0(\SS^3\backslash (\BB(P_j,\rho)\cup \BB(-P_j,\rho))}\leq \frac{C_\rho}\La
\]
for large~$\La$, with a constant that depends on~$\rho$ (and, of
course, on $v$ and $\de$). If we now define
\[
Y_i:=\sum_{j=1}^l \hat Y_{ij}\,,
\]
and choose $\rho$ small enough so that the sets $\BB(P_j,\rho)\cup \BB(-P_j,\rho)$ are disjoint for all $j$, the same reasoning that we employed in the proof of
Proposition~\ref{P.sphharm} shows that
\[
\bigg\|w^i_j-Y_i\circ \Psi_j^{-1}\Big(\frac\cdot \La\Big)\bigg\|_{C^{m+2}(B)}<C\de
\]
for all $1\leq i\leq 3$ and $1\leq j\leq l$, which plays a role
completely analogous to that of Proposition~\ref{P.sphharm} in the
generalized context that we are now considering. The rest of the
argument remains exactly as in Section~\ref{S.sphere}, so the result follows.
\end{proof}

In particular, this yields the existence of Beltrami fields in the
sphere having prescribed sets of closed vortex lines and tubes (modulo
diffeomorphism) around any finite number of points $P_1,\dots,
P_l$. These lines and tubes are contained in balls of radius
$1/\la$. This line of reasoning also allow us to prove an analog of
Theorem~\ref{T.main} in any lens space $L(p,q)$:

\begin{theorem}\label{T.Lpq}
Let $\cS$ be a finite union of (pairwise disjoint, but possibly
knotted and linked) closed curves and tubes contained in a contractible subset of a three-dimensional
lens space $L(p,q)$. Then for any large
enough even integer $\la$ there exists
a Beltrami field~$u$ satisfying the equation $\curl u=\la u$ and a diffeomorphism $\Phi$ of
$L(p,q)$ such that $\Phi(\cS)$ is a union of vortex lines
and vortex tubes of~$u$. Furthermore, this set is structurally stable.
\end{theorem}

\begin{proof}
The lens space can be written as 
\[
L(p,q)=\SS^3/G\,,
\]
where $G$ is a finite isometry group isomorphic to $\ZZ_p$. We can
assume that $G$ is generated by certain isometry~$g$. Let us now fix
a point $p_0\in\SS^3$ and set
\[
P_j:=g^j\cdot p_0
\]
for $0\leq j\leq p-1$. If $\Psi$ is a patch of normal geodesic coordinates around
$p_0$, we will also set $\Psi_j(x):=\Psi(g^{-j}\cdot x)$. Notice that if $p$ is odd there are not any points in the set
$\{P_j\}_{j=0}^{p-1}$ that are antipodal to each other, while for $p$ even
$P_j$ and $P_k$ are antipodal if and only if $|j-k|=\frac p2$. 

Let us fix a Beltrami field~$v$ in $\RR^3$ as in
Theorem~\ref{Acta}. Theorem~\ref{T.approx2} then ensures the existence
of a Beltrami field $\tu$ in $\SS^3$ such that
\[
\bigg\|(\Psi_j)_* \tu\Big(\frac{\cdot}{ \la}\Big)-v_j\bigg\|_{C^m(B)}<\ep\,,
\]
where $0\leq j\leq p'-1$ with $p':= p$ if $p$ is odd and $p':= \frac p2$ if
$p$ is even. Here $v_0:=v$ and $v_j:=0$ for $1\leq j\leq p'-1$. Notice that, as $\la$ is even, we saw in the previous
remark  that $\tu$ is odd, i.e.,
$\tu(x)=-\tu(-x)$, so that $\tu$ is equivariant under the isometry $x\mapsto -x$. Hence, by construction, the vector field
\[
u:=\sum_{j=0}^{p'-1} (g^j)_* \tu
\]
is $G$-equivariant, and therefore it defines a vector field in the quotient space
$L(p,q)=\SS^3/G$ that we still denote by~$u$ with some abuse of notation. Arguing exactly as in the proof
of the main theorem one can show that the vector field~$u$ on $L(p,q)$
indeed has the desired properties, so the statement then follows.
\end{proof}

\subsubsection*{In the torus, the distribution of
    rational points on the 2-sphere is key} The proof that we have given
  holds provided that the eigenvalue~$\la$ is an odd integer of
  sufficiently large absolute value. It does not say anything about
  even integers, or about eigenvalues that are not integers. This
  assertion can be refined a little, however. We have seen that for any eigenvalue $\la$ of the curl operator
  in $\TT^3$ there is a set of points $\{\xi_n\}_{n=1}^N$ lying on the
  unit sphere $\SS^2$ of $\RR^3$ such that $\la\xi_n\in\ZZ^3$ (this is
  obvious from the fact that one can write $\la=|k|$ with
  $k\in\ZZ^3$). Therefore, in the proof of Theorem~\ref{T.approx} for
  the torus (cf.\ Section~\ref{S.torus}) one can substitute the
  collection of odd integers $\La$ by any subset of eigenvalues $\la$
  for which there is a set of points $\{\xi_n\}_{n=1}^N\subset\SS^2$
  (depending on $\la$ and such that the rescalings $\la \xi_n$ are in
  $\ZZ^3$) that becomes dense in the sphere as $|\la|\to\infty$ along
  this subset of eigenvalues. In particular, replacing the density
  condition by the more stringent assumption that $\{\xi_n\}$ becomes
  equidistributed on the sphere, it turns out that the characterization of the numbers
  $\la$ that satisfy this property is somehow related to the celebrated
  Linnik problem in number theory. In particular, since the
  aforementioned equidistribution property holds for any eigenvalue 
  for which the integer $\la^2$ is square-free~\cite{Du88}, we
  immediately infer that the statement of Theorem~\ref{T.main} also
  holds for any large enough eigenvalue $\la$ of curl (possibly even or non-integer) for
  which $\la^2$ is square-free.

\section*{Acknowledgments}

The authors are supported by the ERC Starting Grants~633152 (A.E.) and~335079
(D.P.-S.\ and F.T.L.) and by a fellowship from Residencia de Estudiantes (F.T.L.). This work is supported in part by the
ICMAT--Severo Ochoa grant
SEV-2011-0087.

\bibliographystyle{amsplain}

\end{document}